\documentclass[11pt]{article}
\usepackage{amsfonts}
\usepackage{amssymb}
\usepackage{amsmath}
\usepackage{amsthm}
\theoremstyle{plain}
\newtheorem{thm}{Theorem}[section]

\newtheorem{lem}[thm]{Lemma}

\theoremstyle{definition}
\newtheorem{defn}{Definition}[section]
\newtheorem{exam}{Example}[section]
\theoremstyle{remark}
\newtheorem{rem}{Remark}[section]
\newcommand{\epsi}{\varepsilon}
\numberwithin{equation}{section}
\title{ \bf A Note on the Paper ``Optimality Conditions for Vector Optimization
Problems with Difference of Convex Maps"}
\author{\bf Allahkaram Shafie,$^{1}$ Farid Bozorgnia$^{2}$ \\
$^{1}$Department of Mathematics, Razi University,
Kermanshah, Iran.\\
 $^{2}$Department of Mathematics, Instituto Superior T\'{e}cnico, Lisbon.\\
 E-mail : shafie.allahkaram@gmail.com,  bozorg@math.ist.utl.pt}
\date{}

\begin{document}
\setlength{\evensidemargin}{-0.2in}
\maketitle
%\begin{center}\setlength{\columnsep}{1.5cm}

%{\bf\large Abstract}\\
%\end{center}
\begin{abstract}
\noindent In this work, some counterexamples are given  to refute  some results reported in the  paper by Guo and Li  \cite{GL} (J Optim Theory  Appl 162,(2014), 821-844). We correct the faulty in some of their   theorems  and we present alternative proofs. Moreover, we extend  the definition  of approximately pseudo-dissipative in the setting of metrizable topological vector spaces.  
\end{abstract}
\vspace{2mm}

 \section{Introduction}
In several optimization problems   nonlinear and   nonconvex
functions can be decomposed  into the  difference of convex (DC) functions (see \cite{Tuy}).

\hspace*{-0.6cm}In the last decade, different kinds of DC programming have been investigated extensively and significant  results have been achieved,  see for example  \cite{DMN2,DMN1,FLN,FWZ,FZ,GL,Jey,LYZ,SGZ,Tuy,Vo} and the references therein. Here, we briefly  mention the results on duality and optimality  in
\cite{DMN2,DMN1,NGA,FZ,GL,SGZ,SLZ}.   In  \cite{DMN2,DMN1,NGA} the authors consider  optimization problems with objectives given as DC functions and constraints described by  convex inequalities.   For Banach spaces, they  obtain  necessary and sufficient optimality conditions for DC
infinite and semi-infinite programs. Efficient upper estimates of certain subdifferentials of   value functions for the DC optimization problem are given in \cite{DMN2}.  In \cite{DMN1},  the authors provide  characterizations of the Farkas-Minkowski constraint qualification.

 \hspace*{-0.6cm}Fang and Zhao    introduced the local and global KKT type conditions for the DC optimization problem  in \cite{FLN}. Using   properties of the  subdifferential, they provide   some sufficient and  necessary conditions for these   optimality conditions.
In the case of  DC optimization, weak and
strong duality assertions for extended Ky Fan inequalities are provided  in  \cite{SGZ}. The authors in \cite{SGZ}  apply their dual
problems also  to a convex optimization problem and a generalized variational inequality problem. By using the
properties of the epigraphs of the conjugate functions, Sun, et al. \cite{SLZ} introduced a closedness qualification condition. They then employed their condition to investigate duality and Farkas-type results for a DC infinite programming problem.  Also in \cite{LA} established optimality conditions under convexity and continuity assumptions for set functions.

In \cite{GL}, Guo and Li   use the notions of strong subdifferential and epsilon subdifferential to obtain
necessary and sufficient optimality conditions for an epsilon-weak Pareto minimal point and an
epsilon-proper Pareto minimal point of a DC vector optimization problem.

\hspace*{-0.7cm} In this article, we show that some  theorems and results in \cite{GL} are not correct. Furthermore, we clarify   an existence gap  by providing some counterexamples. Finally  we  present corrected versions of their results.

\section{Preliminaries}\label{section 2}

Let us briefly recall   the notation used in  this work. For the most part, we follow notations as in \cite{GL}.   Throughout this paper, $X$ is a  metrizable topological vector space. Furthermore, $Y$ and $Z$ stand for topological vector spaces. We will denote the dual of $Y$ and $Z$ by $Y^*$ and $Z^*$ respectively, with duality pairing
denoted by $\left\langle {.,.} \right\rangle.$ The origins of the topological vector spaces are denoted by $0_X, 0_Y$, and $0_Z.$  As usual   $L(X, Y )$  is the set of all linear continuous operators from
$X$ to $Y$.    Moreover, let $K \subset Y$  and $D \subset Z$  be  proper (i.e., $K\ne\{0_Y\} \ne Y$ ) convex cones with nonempty interior (i.e., $ {\rm int}K\ne \emptyset$). Let $l(K) = K \cap -K$ be the linearity of $K.$
The cone $K$ determines an order relation on $Y$ denoted in the sequel
by $\mathop\preceq_K.$ We recall the following definition  of  ordering relations:
$$\begin{array}{l}
 y'\mathop\preceq_K y \Leftrightarrow y - y' \in K, \\
 y'\mathop\prec_K y \Leftrightarrow y - y' \in {\mathop{\rm int}} K, \\
 y'\mathop\npreceq_K y \Leftrightarrow y - y' \notin K, \\
 y'\mathop\nprec_K y \Leftrightarrow y - y' \notin {\mathop{\rm int}} K. \\
 \end{array}$$
 The negative polar cone(or dual cone) $K^*$ of $K$ and the strict polar cone $(K^*)^\circ$ of $K$ are defined respectively by
 $$K^*=\{y^*\in Y^*:\langle y^*,y \rangle\ge 0~\text{for all}~ y\in K\},$$ and $$(K^*)^\circ=\{y^*\in Y^*:\langle y^*,y \rangle> 0~\text{for all}~ y\in K\setminus l(K)\}.$$
 Clearly,
  $(K^*)^\circ \subset K^* \setminus \{0\}$ since $K+K\setminus l(K)=K\setminus l(K).$
  For  $A \subset X$ the indicator function   $\delta_A:X\longrightarrow\mathbb{R}\cup\{+\infty\}$  is defined by
  $${\delta _A}(x) = \left\{ \begin{array}{ll}
 0   & x \in A,\\
 +\infty  &  x \notin A.\\
 \end{array} \right. $$

 \begin{rem}
Note that in  a locally convex space  $Y,$ always there exists a convex cone with nonempty interior. Indeed,  if $U$ be a convex neighborhood of zero and $y\notin Y,$ then   it is sufficient to consider $K={\rm cone}(U-y)\subset Y.$
\end{rem}

 \begin{defn}
The vector-valued map $F : X\longrightarrow Y$ is said to be $K$-convex iff, for all $x_1, x_2\in X$ and $0\leq\lambda\leq1$, the following inequality $$F(\lambda {x_1} + (1 - \lambda ){x_2}) \preceq_K \lambda F({x_1}) + (1 - \lambda )F({x_2}),$$
holds. Also $F$ is said to be $K$-convexlike iff for all $x_1,x_2\in X$ and $0\leq\lambda\leq1$ there exists $x_3\in X$ such that $$F({x_3}) \preceq_K \lambda F({x_1}) + (1 - \lambda )F({x_2}).$$
\end{defn}
It is worth to mention  that $F$ is $K$-convexlike on a convex subset  $C\subset X$ iff $F(C)+K$ be convex.

\begin{defn}
Let $X$ and $Y$ be topological linear spaces, $Y$ be ordered by a convex cone $K\subset Y,$ and $F:X\longrightarrow Y$ be a given map. For an arbitrary $\bar{x}\in X,$ the set
 $$\partial F(\overline x ): = \left\{ {T \in L(X,Y)|\,\,T(x - \overline x )  \preceq_K F(x) - F(\overline x ), \, \,\,\forall x \in X} \right\}$$
\end{defn}
is called the  strong  subdifferential  of $F$ at $ \overline{x}.$ Also  let $\varepsilon \in K,$ then  $\varepsilon$-subdifferential of $F$ at  $\bar{x}$ is defined as following
$${\partial _{\varepsilon {\mkern 1mu} {\mkern 1mu} }}F(\bar x): = \left\{ {T \in L(X,Y)|{\mkern 1mu} {\mkern 1mu} T(x - \overline x ){\preceq _K}F(x) - F(\bar x) + \varepsilon, {\mkern 1mu} {\mkern 1mu} \,\,\forall x \in X} \right\}.$$

We consider the following cone-constrained vector optimization problem as in   \cite{GL}  sometimes called $DC$ vector optimization where  refers to difference of two cone convex functions:
\begin{equation*}
	(P)~~
	\left\{ {\begin{array}{*{20}{c}}
   K- {\mathop  {\textrm{Min}}\limits {\rm{ }}\left( {  F(x) - G(x)} \right),} \hfill  \\
   {{\rm{subject ~to}} ~x \in C\,{\rm{and}}\,H(x) - S(x) \in -D,} \hfill  \\
\end{array}} \right.
\end{equation*}
where $F, G:X\longrightarrow Y$ are $K$-convex and
$S, H:X\longrightarrow Z$ are $D$-convex maps and $C$ is a convex subset of $X.$

\begin{defn}\cite{GL}
	Suppose that $\Omega:=\{x\in C: H(x)-S(x)\in -D\}$ and $\varepsilon\in K$. An element $\bar{x}\in\Omega$ is called an
	$\varepsilon$-weak local Pareto minimal solution of problem ($P$) iff
	there exists a neighborhood $U$ of $\bar{x}$ such that
	\[
F(\bar{x})-G(\bar{x})\in \epsi \mbox{WMin}(F-G){(U\cap\Omega)},
\]
	i.e.,
	$$(F-G){(U\cap\Omega)}\subset F(\bar{x})-G(\bar{x})-\epsilon+Y\setminus -{\rm int}K,$$ where
\[
(F-G){(U\cap\Omega)}=\{F(x)-G(x): x\in U\cap\Omega\}.
\]
Similarly, $\bar{x}$ is said to be an $\varepsilon$-proper local Pareto minimal solution of problem $(P)$ iff there exists a neighborhood $U$ of $\bar{x}$ such that
\begin{equation*}
F(\overline x ) - G(\overline x ) \in \epsi \mbox{PMin}(F - G)(U \cap \Omega ), 
\end{equation*}

i.e., there exists a convex cone $K^{'}\subset Y$ with $K\setminus l(K)\subseteq {\rm int} K^{'}$ such that
$$(F - G)(U \cap \Omega ) \subset F(\overline x ) - G(\overline x ) - \varepsilon  + Y\backslash  - {\mathop{\rm int}} {K^{'}}.$$
\end{defn}

%%%%%%%%%%%%%%%%%%%%%%%%%%%%%%%%%%%%%%%%%%%%%

In the sequel we use  the following well-known property, see \cite{JJ}.
 \begin{lem}\label{l1}Let $K$ be a convex cone in  topological vector space $Y.$ Then the following assertion holds
\begin{eqnarray}
y\in{\rm int}K\Rightarrow \left\langle  y^*,y \right\rangle >0 ,~~~~\forall  y^*\in K^*\setminus\{0\}.
\end{eqnarray}
\end{lem}

The following definition is based on metrizable topological vector space which is  slightly different from Definition 3.1 in \cite{GL}. We note that $X$ with the topology generated by  metric $d$ is a  topological vector space.
\begin{defn}
A set valued $M: X\rightrightarrows L(X,Y)$ is said to be approximately
pseudo-dissipative at $\bar{x}$  iff, for every $\epsilon\in {\rm int}K$, one can find a neighborhood
$U$ of $\bar{x}$ such that
\begin{equation}\label{q1}
\forall x \in U, \,\,\,\exists T \in M(x),\,\,{T^*} \in M(\overline{x} )\,\,\,\text{s.t.}\,\,{\left( {T - {T^*}} \right) (x - \overline x) } \preceq_K\varepsilon d(x,\overline x ).
\end{equation}
\end{defn}
\section{Sufficient optimality condition}\label{section 3}

In this part, first  we review  the Theorems 3.1 and 3.2 stated in  \cite{GL}, then we give an  example which demonstrates   these theorems are not correct.

(Theorem 3.1  \cite{GL}) Let  $\overline{x} \in \Omega$. Assume that the set-valued maps  $ \partial_{\epsilon} G$   and  $\partial S $  are both
approximately pseudo-dissipative at  $\overline{x}$. If in addition, for any
 $ T\in \partial_{\epsilon} G(\bar{x})$   and
$ L \in \partial S(\bar{x}),$  there exist  $y^{*} \in K^{*}\setminus{\{0}\} $ and $z^{*} \in D^{*}$ such that
\begin{equation*}
\left\{ \begin{array}{ll}
 \,\,{y^*}o T  + \,{z^*}o L \in \partial {({y^*}o F + \,{z^*}o H){(\overline x )}}, & \\
\left\langle {{z^*},H(\overline x ) - S(\overline x )} \right\rangle  = 0,\\
 \end{array} \right.
\end{equation*}
then $\bar{x} $  is an $\epsilon$-weak local Pareto minimal solution of problem (P).

(Theorem $3.2$, \cite{GL})
Let $\bar{x}\in\Omega.$ Assume that the set-valued maps $\partial_{\varepsilon}G$ and $\partial S$ are both approximately pseudo-dissipative at $\bar{x}.$ If in addition, for any
 $ T   \in {\partial _\varepsilon }G(\overline x )   $ and
 $ L   \in    \partial S(\overline{x} )$  there exist $y^* \in (K^*)^{\circ}$   and  $ z^*\in D^*$  such that
\begin{equation*}
\left\{ \begin{array}{l}
 \,\,{y^*}o T  + \,z^{*}o L  \in \partial {(y^{*}o F + \,z^{*}o H)(\overline{x} )}, \\
 \left\langle z^{*},H(\overline x ) - S(\overline x )\right\rangle  = 0, \\
 \end{array} \right.
\end{equation*}
then $\bar{x}$ is an $\varepsilon$-proper local Pareto minimal solution of problem $(P).$

The following example shows that Theorems 3.1  and 3.2 and subsequent corollaries 3.1, 3.2, 3.3, 3.4, 3.5, 3.6 in \cite{GL} are not correct, and need several corrections.
\begin{exam}
Let $X=\mathbb{R},Y=Z=\mathbb{R}^2,C=[-1,1],K=D=[0,+\infty)\times [0,+\infty),\bar{x}=0,\varepsilon=(0,0).$ Define  $F,G,H,S:\mathbb{R} \to \mathbb{R}^2$ by
\begin{equation*}
\,\left\{ \begin{array}{l}
 F(x) = {(x^4,x^2)} \\
 G(x) = {(x^2,2x^2)} \\
 H(x) = (x,-1 )\\
 S(x) =(x + 1,0). \\
 \end{array} \right.
\end{equation*}
Clearly $F,G$ are $K$-convex and $H, S$ are $D$-convex and
\[
 \Omega=\left\{ {x \in C:\,H(x) - S(x) \in  - D} \right\} = [- 1,1].
  \]
Also we have
    \[
    \partial {G_\varepsilon }(x)=\{(2x,4x)\} \quad  \textrm{and} \quad  \partial S( x )=\{(1,0)\}.
     \]
   Since  $ \partial {G_\varepsilon }, \partial S$ are continuous then by Lemma 3  in  \cite{PN}   are approximately pseudo-dissipative at $\bar{x}=0$. For  given $T\in \partial {G_\varepsilon }( \bar{x} )$ and $L\in \partial S( \bar{x}),$   we let  $z^*=0, y^*\in K^*\setminus\{0\}.$ One can easily check  that
 $$\left\{ \begin{array}{l}
 \left\langle {{z^*},H(\bar x) - S(\bar x)} \right\rangle  = \left\langle {{z^*},( - {\rm{ }}1, - 1)} \right\rangle  = 0, \\
 {y^*}oT + {\rm{ }}{z^*}oL = 0 \in \partial {\left( {{y^*}oF + {\rm{ }}{z^*}oH} \right)(\bar x)} = \partial {\left( {{y^*}o({x^4},{x^2})} \right)(0)}. \\
 \end{array} \right.$$

 Observe that all  hypotheses of Theorem 3.1 in  \cite{GL} are satisfied, but $\bar{x}$ is not an $\varepsilon$-weak local Pareto minimal solution of problem (P). Indeed,  for any neighborhood  $U$ of $\bar{x}=0$ and $x\in U\cap\Omega,$ one has $$F(x) - G(x) - \left( {F(\overline x ) - G(\overline x )} \right) = (x^4 - x^2,-x^2 )\in  - {\mathop{\rm int}} K=(-\infty,0)\times(-\infty,0).$$
\end{exam}

 The following theorems  are modifications of Theorems 3.1 and 3.2   in \cite{GL} respectively.
\begin{thm}
Let $\bar{x}\in\Omega.$ Assume that the set-valued maps $\partial_{\varepsilon}G$ and $\partial S$ are both approximate pseudo-dissipative at $\bar{x}.$ If in addition, for any $\left( {T,L} \right) \in {\partial _\varepsilon }G(\overline x ) \times \partial S(\overline x )$ and $\left( {\alpha ,\beta } \right) \in {\mathop{\rm int}} K \times {\mathop{\rm int}} D$ there exist $(y^*,z^*)\in K^*\backslash\{0\}\times D^*$ such that
\begin{equation*}
\left\{ \begin{array}{l}
 \,\,{y^*}o (T - \alpha ) + \,{z^*}o (L - \beta ) \in \partial {({y^*}o F + \,{z^*}o H){(\overline x )}}, \\
\left\langle {{z^*},H(\overline x ) - S(\overline x )} \right\rangle  = 0,\\
 \end{array} \right.
\end{equation*}
then $\bar{x}$ is an $\varepsilon$-weak local Pareto minimal solution of problem $(P).$
\end{thm}
\begin{proof}
By approximately pseudo-dissipativity of $\partial_{\varepsilon}G$  and $\partial S$  at $\bar{x},$ for given $\alpha\in {\rm int}K$ and $\beta\in {\rm int}D$ there exist  neighborhoods $V_{\alpha}$ and $V_{\beta}$ of $\bar{x}$ such that (\ref{q1}) holds  for $\partial_{\varepsilon}G$ and $\partial S.$  Let $V=V_{\alpha}\cap V_{\beta}.$ Hence
\begin{equation}\label{r1}
\begin{array}{l}
 \forall x \in V,\,\,\exists \left( {T^{'},T} \right) \in {\partial _\varepsilon }G(x) \times \partial G(\overline{x} )\,,\,\,\left( {L^{'},L} \right) \in \partial S(x) \times \partial S(\overline{x} )\,\,\, \\
 \text{such that}\,\,
 \left\{ \begin{array}{l}
 {\left( {T^{'} - T} \right)(x - \overline{x} )} \preceq _{K}\alpha d(x,\overline{x}), \\
 {\left( {L^{'} - L} \right)(x - \overline{x} )} \preceq _{D}\beta  d(x,\overline{x} ). \\
 \end{array}
  \right\}. \\
 \end{array}
\end{equation}

We claim that for all $x\in V\cap\Omega,$ there exist $y^*\in K^*\setminus\{0\}$ and $z^*\in D^*$ such that
\begin{equation}\label{j}
\begin{array}{l}
  \langle {y^*},F(x) - G(x) - (F(\overline x ) - G(\overline x )) + \varepsilon  \rangle  \\
  +  \langle {y^*},\alpha (d(x,\overline x ) - 1) \rangle  + \langle{z^*},\beta (d(x,\overline x ) - 1) \rangle \ge 0.  \\
 \end{array}
\end{equation}

Fix $x\in V\cap\Omega.$ Then by (\ref{r1}) there exists
 $T^{'}\in \partial_{\varepsilon}G(x)$ and $L^{' }\in\partial S(x),$ such that   $\forall y \in X$ the  following hold
\begin{equation}\label{r2}
\left\{
 \begin{array}{l}
 G(y) - G(x) - T^{'}(y - x) + \varepsilon  \in K,  \\
 S(y) - S(x) - L^{'}(y - x) \in D.    \\
 \end{array}
  \right.
\end{equation}

Next let $y=\bar{x}$, we get
\begin{equation}\label{r3}
\left\{ \begin{array}{l}
 G(\overline{x} ) - G(x) - T^{'}(\overline{x}  - x) + \varepsilon  \in K, \\
 S(\overline{x} ) - S(x) - L^{'}(\overline{x}  - x) \in D. \\
 \end{array} \right.\,\,\,\,\,\,\,\,\,\,\,
\end{equation}
Since $T\in \partial_{\varepsilon}G(\bar{x})$ and $L\in \partial S(\bar{x}),$ by the assumption there exists $(y^*,z^*)\in K^*\backslash\{0\}\times D^*$ such that
\begin{equation}\label{r4}
\left\{ \begin{array}{l}
 \,\,y^{*}o(T - \alpha ) + \,z^{*}o(L - \beta ) \in \partial {(y^{*}oF + \,z^{*}oH){(\overline {x} )}}, \\
 \left\langle {z^{*},H(\overline {x} ) - S(\overline {x} )} \right\rangle  = 0. \\
 \end{array} \right.
\end{equation}

Therefore
\begin{equation}\label{r5}
\begin{array}{l}
 \left\langle {{y^*},F(x) - F(\bar x) - T(x - \bar x)} \right\rangle  + \left\langle {{z^*},H(x) - H(\bar x) - T(x - \bar x)} \right\rangle  \\
  - \left\langle {{y^*},\alpha } \right\rangle  - \left\langle {{z^*},\beta } \right\rangle  \ge 0. \\
 \end{array}
\end{equation}
By  using the fact that  $y^*\in K^{*}, z^*\in D^{*},$ and  (\ref{r3}) we deduce that
\begin{equation}\label{r6}
\left\{ \begin{array}{l}
\langle {y^{*},G(\overline {x} ) - G(x) -  T^{'}(\overline {x}- x) + \varepsilon } \rangle  \ge 0, \\
\langle {{z^*},S(\overline {x} ) - S(x) -    L^{'}(\overline {x}- x) + \varepsilon } \rangle  \ge 0.\\
 \end{array} \right.\,\,\,\,\,
\end{equation}

From  (\ref{r5}) and (\ref{r6})  we obtain that
\begin{equation}\label{r7}
\begin{array}{l}
 \langle {y^{*},F(x) - G(x) - ( {F(\overline {x} ) - G(\overline {x} )} ) - {{( T^{'} - T)}{(x - \overline {x} )}} + \varepsilon } \rangle  \\
  + \langle {{z^*},H(x) - S(x) -( {H(\overline {x} ) - S(\overline x )} ) - {{( {L^{'} - L} )}{(x - \overline x )}}} \rangle \\  -\langle {y^{*},\alpha } \rangle  -\langle {{z^*},\beta } \rangle  \ge 0. \\
 \end{array}
\end{equation}
Form     $H(x)-S(x)\in -D$ we have
\[
\left\langle {{z^*},H(x) - S(x)} \right\rangle  \le 0.
\]
 In addition, using $\left\langle {{z^*},H(\bar{x}) - S(\bar{x})} \right\rangle  = 0$ we   get
\begin{equation}\label{r8}
\begin{array}{l}
 \langle y^{*},F(x) - G(x) - (F(\overline{x} ) - G(\overline {x} )) + \varepsilon \rangle  -
 \langle y^{*},( T^{'} - T ){(x - \overline x )} \rangle  \\
  - \langle z^{*},(  L^{'} - L ){(x - \overline {x} )} \rangle
  - \langle y^{*},\alpha  \rangle  - \langle  z^{*},\beta  \rangle  \ge 0, \\
 \end{array}
\end{equation}
by using (\ref{r1}) and $(y^{*},z^{*})\in K^{*}\setminus\{0\} \times D^{*},$ we obtain  that
\begin{equation}\label{r9}
\left\{
 \begin{array}{l}
 \langle y^{*},\alpha d(x,\overline {x} ) - (T^{'} - T){(x - \overline{x} )} \rangle  \ge 0, \\
 \langle z^{*},\beta  d(x,\overline {x} ) - (L^{'} - L){(x - \overline{x} )} \rangle  \ge 0. \\
 \end{array}
 \right.
\end{equation}
Next by   combining  (\ref{r9}) and (\ref{r8}) the following holds
\begin{equation}\label{r10}
\begin{array}{l}
  \langle y^{*},F(x) - G(x) - (F(\overline{x} ) - G(\overline{x} )) + \varepsilon  \rangle  \\
  +  \langle y^{*},\alpha (d(x,\overline {x} ) - 1) \rangle
  + \langle z^{*},\beta (d(x,\overline{x} ) - 1) \rangle \ge 0.  \\
 \end{array}
\end{equation}

This completes  the proof of (\ref{j}). Next,  $X$ is metrizable, so  there exists a neighborhood $U\subseteq V$ of $\bar{x}$ such that for all $y\in U$ we have $d(y,\bar{x})\le 1.$  Assume  that $y\in U\cap\Omega\subseteq V\cap\Omega$ be given, so there exists $y^*\in K^*\setminus\{0\},z^*\in D^*$ such that (\ref{j}) holds for $x=y$. On the other hand, using  $\alpha\in {\rm int}K,\beta\in {\rm int}D$  follows that
\begin{equation}\label{rz}
\langle y^{*},\alpha (d(y,\overline {x} ) - 1) \rangle  \le 0\,\,\,\,\,\,\text{and}\,\,\,\,\,\,\, \langle z^{*},\beta (d(y,\overline {x} ) - 1) \rangle \leq 0.
\end{equation}
 Combining (\ref{j}) with (\ref{rz}), yields
\begin{equation}
 \langle y^{*},F(y) - G(y) - (F(\overline{x} ) - G(\overline{x} )) + \varepsilon  \rangle  \ge 0.\,\,\,\,\,\,\,\,\,
\end{equation}
Finally by   Lemma \ref{l1} one has
\begin{equation*}
F(y) - G(y) - (F(\overline{x} ) - G(\overline{x} )) + \varepsilon  \notin  - {\mathop{\rm int}} K,
\end{equation*}
but since $y\in U\cap\Omega$ was arbitrary,  thus  $\bar{x}$ is a $\varepsilon-$weak local Pareto minimal solution of problem $(P).$ This complete the proof.
\end{proof}

By  similar argument  as the previous theorem, we can obtain the following theorem for sufficient optimality condition.
\begin{thm}
Let $\bar{x}\in\Omega.$ Assume that the set-valued maps $\partial_{\varepsilon}G$ and $\partial S$ are both approximately pseudo-dissipative at $\bar{x}.$ If in addition, for any
 $\left( {T,L} \right) \in {\partial _\varepsilon }G(\overline x ) \times \partial S(\overline{x} )$ and $\left( {\alpha ,\beta } \right) \in {\mathop{\rm int}} K \times {\mathop{\rm int}} D$ there exist $(y^*,z^*)\in (K^*)^{\circ}\times D^*$ such that
\begin{equation*}
\left\{ \begin{array}{l}
 \,\,{y^*}o (T - \alpha ) + \,z^{*}o (L - \beta ) \in \partial {(y^{*}o F + \,z^{*}o H){(\overline{x} )}}, \\
 \left\langle z^{*},H(\overline x ) - S(\overline x )\right\rangle  = 0, \\
 \end{array} \right.
\end{equation*}
then $\bar{x}$ is an $\varepsilon$-proper local Pareto minimal solution of problem $(P).$
\end{thm}

\section{Necessary Optimality Conditions  }\label{section 4}
In this section, we  provide sufficient optimality conditions for an $\varepsilon$-weak Pareto minimal
solution and an $\varepsilon$-proper Pareto minimal solution for the cone-constrained vector
optimization problem (P).   Here the  objective function and constraint set are  given as  differences of two vector-valued maps. Our results are corrections of Theorems 4.1 and 4.2 in \cite{GL}. 

\begin{thm}\cite{GL} 
	Let $\varepsilon  \in K $ and  $\bar{x}\in\Omega$. If the vector-valued map $F:X\longrightarrow
	Y$ is a $K$-convex  map, the vector-valued map $H:X\longrightarrow Z$
	is a $D$-convex  map, and $\bar{x}$ is an $\varepsilon$-proper
	local minimal solution of $(P)$, then there exist $y^{*}\in (K^{*})^{\circ}\cup\{0\}$ and $z^{*}\in D^{*}$ and $(y^{*},z^{*})\neq (0_{Y^{\ast}},0_{Z^{\ast}})$ such that
	 \begin{equation*}
\left\{ \begin{array}{ll}
		(y^{*}o\partial  G +z^{*}o\partial H){(\bar{x})}\cap
		\partial_{\langle y^{*},\varepsilon\rangle}(y^{*}o F+z^{*}o H+\delta_{U\cap C}){(\bar{x})}, &\\
 \left\langle z^{*},H(\overline{x} ) - S(\overline{x} ) \right\rangle=0,  
  \end{array} \right.
\end{equation*}
	
where $U$ is a neighborhood of $\bar{x}.$
\end{thm}
The following example shows that  Theorems 4.1 and 4.2 and subsequent  corollaries in \cite{GL} are not correct.
\begin{exam}
Take $X=\mathbb{R},Y=Z=\mathbb{R},C=[-1,1],K=D=[0,+\infty) ,\bar{x}=0,\varepsilon=0.$ Consider  $F,G,H,S:\mathbb{R} \to \mathbb{R}$
 defined by
 $$F(x) =
 \left
 \{\begin{array}{l}
  - 1\,\,\,\,\,\,\,x \ne 0, \\
 0\,\,\,\,\,\,\,\,\,\,x = 0. \\
 \end{array} \right.
 \,\,\,\,\,\,\,G(x) =
  \left\{\begin{array}{l}
  - 2\,\,\,\,\,x \ne 0, \\
 0\,\,\,\,\,\,\,\,x = 0. \\
 \end{array} \right.\,\,\,\,\,H(x) = x - 1,\,\,\,S(x) = x.$$
 Clearly $F,G$ are $K$-convex and $H,S$ are $D$-convex, with $\partial G (\overline x ) = \{0\},\partial H(\overline x ) = \{1\}.$ One can verify that
 $$
  \Omega=\{x\in C:~H(x)-S(x)\in -D\}=[-1,1].
   $$
   Hence for a neighborhood $U$ of $\bar{x}=0$ we have
    $$F(x) - G(x) - \left(F(\overline{x}) - G(\overline{x}) \right)+\varepsilon \notin  - \mathop{{\rm int}} K,\,\,\,\,\forall x \in U \cap\Omega,$$ which implies that $\bar{x}=0$ is $\varepsilon$-weak local minimal solution of $(P)$. If
  $$\left
  \{ \begin{array}{l}
 ({y^*}o \partial G + {z^*}o \partial H)(\bar x) \subset {\partial _{\langle {y^*},\varepsilon\rangle }}({y^*}o F + {z^*} o H + {\delta _{U \cap C}})(\bar x), \\
 \left\langle {{z^*},H(\bar x) - S(\bar x)} \right\rangle  = 0, \\
 \end{array}
  \right.$$
then
     $$\left\langle z^{*},H(\overline x ) - S(\overline{x} ) \right\rangle  = \left\langle z^{*}, -1  \right\rangle  = 0,$$
  which implies that   $z^* = 0.$     Therefore one has
\[
\left(y^{*} o \partial G + z^{*}o\partial H  \right)(\bar{x}) = 0 \in \partial \left(y^{*}oF  + \delta_{U \cap C} \right)(\bar{x}),
\]
which gives $F(x)\ge 0$ for all $x\in U\cap\Omega,$ that is contradiction.
\end{exam}
We generalize the result (Theorem 3.3 in \cite{LL})  Farkas-Minskowski   for $D$-convexlike single value functions.
\begin{lem}\label{l2}
	Let $C$ be a convex subset of $X$. If the map $F:C\longrightarrow
	Y$ is $K$-convexlike and $G:C\longrightarrow Z$ is $D$-convexlike and the
	system \[
 \, \left\{
	\begin{array}{l}
	F(x)\in -{\rm int} K, \\
	G(x)\in -{\rm int} D, \\
	\end{array}
 \right.
	\]
	has no solution in $C$, then there exist $(y^{*},z^{*})\in K^{*}\times
	D^{*}$  with $(y^{*},z^{*})\neq (0,0)$,  such that
	$$\langle y^{*},F(x)\rangle+\langle z^{*},G(x)\rangle\geq 0~~~~\forall x\in C.$$
\end{lem}

\begin{proof}
We can easily prove that $F(C)+K$ and $G(C)+D$ are convex sets. Define the set valued map $g:C\rightrightarrows X\times Y$ by
$$\begin{array}{l}
  g(x) = (F(x) + K) \times (G(x) + D).
 \end{array}
$$
One  can check  that $g(C)\cap {\rm int} ((-K)\times (- D))=\emptyset.$  Next,   $g(C)$ is convex set  hence,  by the separation theorem, there exist a  non zero $(y^*,z^*)\in  K^*  \times D^*$  and $ \alpha \in \mathbb{R}$ such that for all $(k,d) \in (K,D), x\in C$ we have
$$\left\langle {{y^*}, - k} \right\rangle  + \left\langle {{z^*}, - d} \right\rangle  \le \alpha  \le \left\langle {{y^*},F(x) + k} \right\rangle  + \left\langle {{z^*},G(x) + d} \right\rangle. \,\,\,\,\,\,\,\,\,$$
Choosing $ k=d=0$  yields
 $$\langle {y^*},F(x)\rangle  + \langle {z^*},G(x)\rangle  \ge 0{\rm{ }},  \,\,\,\,\,\,\,\,\,\,\,\forall x \in C.$$
\end{proof}
In the  rest  of this section,  we present  modification of Theorem 4.1 and 4.2 (Necessary optimality conditions) in \cite{GL} by assuming  convex-like  condition which is   weaker than convexity.
\begin{thm}\label{t3}
	Let $\bar{x}\in\Omega$. If the vector-valued map $F:X\longrightarrow  Y$ is a $K$-convexlike map, the vector-valued map $H:X\longrightarrow Z$
	is a $D$-convexlike map, and $\bar{x}$ is an $\varepsilon$-weak
	local minimal solution of $(P)$, then there exist $(y^{*},z^{*})\in K^{*}\times D^{*}$ and $(y^{*},z^{*})\neq (0_{Y^{\ast}},0_{Z^{\ast}})$ such that
	\begin{eqnarray}\label{r11}&
		(y^{*}o\partial G+z^{*}o\partial H){(\bar{x})}\cap
		\partial_{\langle y^{*},\varepsilon\rangle}(y^{*}o F+z^{*}o H+\delta_{U\cap C}){(\bar{x})}\ne\emptyset,
	\end{eqnarray}
where $U$ is a neighborhood of $\bar{x}.$
	\end{thm}
	\begin{proof} Let $\bar{x}\in\Omega$ and $\varepsilon\in K$. Since $\bar{x}$ is
		an $\varepsilon$-weak local minimal solution of (P), there exists a
		neighborhood $U$ of $\bar{x}$ such that for all $x\in U\cap C,$
		$$F(x)-G(x)-(F(\bar{x})-G(\bar{x}))+\varepsilon\notin -{\rm int}K.$$

		Now suppose that $T\in\partial G(\bar{x})$ and $L\in\partial H(\bar{x})$
		be   arbitrary elements. Note that  $F$ is
		$K$-convexlike and $G$ is  $D$-convexlike mapping, thus
		  $F(\cdot)-F(\bar{x})-T(\cdot-\bar{x})+\varepsilon$ is
		$K$-convexlike mapping and $H(\cdot)-H(\bar{x})-L(\cdot-\bar{x})$ is  $D$-convexlike
		mapping.  We prove that the system
		\begin{equation}\label{r12}
			\begin{cases}

				F(x)-F(\bar{x})-T(x-\bar{x})+\varepsilon\in -\mbox{int}K\\

				H(x)-H(\bar{x})-L(x-\bar{x})\in -\mbox{int}D,

			\end{cases}
		\end{equation}

has no solution in $U\cap C$. Arguing by contradiction, assume that
		there exists a solution $x_0\in U\cap C$ of \eqref{r12}. Thus
		\begin{equation}\label{r13}
			\begin{cases}

				F(x_0)-F(\bar{x})-T(x_0-\bar{x})+\varepsilon\in -{\rm int}K,\\
				H(x_0)-H(\bar{x})-L(x_0-\bar{x})\in -{\rm int}D.
			\end{cases}
		\end{equation}

		Since $T\in\partial G(\bar{x})$ and $L\in\partial S(\bar{x}),$ we
		have $$G(x)-G(\bar{x})-T(x-\bar{x})\in K~~\forall x\in X,$$ and
		$$S(x)-S(\bar{x})-L(x-\bar{x})\in D~~\forall x\in X.$$
		Let $x=x_0$,thus  one has
		
		\begin{equation}\label{r14}
 \begin{cases}
  - G({x_0}) + G(\bar x) + T({x_0} - \bar x) \in  - K, \\
  - S({x_0}) + S(\bar x) + L({x_0} - \bar x) \in  - D. \\
 \end{cases}
		\end{equation}
		Next,  note that
		\[
		-K-{\rm int} K=-{\rm int}K,\,\, \,   -D-{\rm int}D=-{\rm int}D, \,\,\,
		H(\bar{x})-S(\bar{x})\in -D.
		\]
		 Then combining (\ref{r13}) with (\ref{r14}), gives us that	
		\begin{equation*}
			\begin{cases}

		F(x_0)-G(x_0)-(F(\bar{x})-G(\bar{x}))+\varepsilon\in -{\rm int}K,\\
				H(x_0)-S(x_0)\in -{\rm int}D,
			\end{cases}
		\end{equation*}

		this contradicts the assumption $\bar{x}$ is an
		$\varepsilon$-weak local minimal solution of (P). Hence,   the system
		\eqref{r12} has no solution. Now by Lemma \eqref{l2}
		there exists $(y^{*},z^{*})\neq (0,0)$ such that for all $x\in U\cap
		C,$

		$$\langle y^{*},F(x)-F(\bar{x})-T(x-\bar{x})+\varepsilon\rangle+\langle z^{*},H(x)-H(\bar{x})-L(x-\bar{x})\rangle\geq 0.$$
 Consequently,
		$$(y^{*} o F+z^{*}o  H)(x)-(y^{*}o F+z^{*} o H)(\bar{x})+\langle y^{*},\varepsilon\rangle-(y^{*} o T+z^{*}o L){(x-\bar{x})}\geq 0.$$
		Hence it follows that
 $$({y^*}oT + {z^*}oL)(\bar x) \in {\partial _{\langle {y^*},\varepsilon\rangle }}({y^*}oF + {z^*}oH + {\delta _{U \cap C}})(\bar x).$$
This completes the proof.
	\end{proof}
By similar proof of the previous theorem we can obtain the following theorem for necessary optimality condition.
\begin{thm}\label{t4}
	Let $\bar{x}\in\Omega$. If the vector-valued map $F:X\longrightarrow
	Y$ is a $K$-convexlike map, the vector-valued map $H:X\longrightarrow Z$
	is a $D$-convexlike map, and $\bar{x}$ is an $\varepsilon$-proper
	local minimal solution of $(P)$, then there exist $y^{*}\in (K^{*})^{\circ}\cup\{0\}$ and $z^{*}\in D^{*}$ and $(y^{*},z^{*})\neq (0_{Y^{\ast}},0_{Z^{\ast}})$ such that
	\begin{eqnarray}\label{r11}&
		(y^{*}o\partial  G +z^{*}o\partial H){(\bar{x})}\cap
		\partial_{\langle y^{*},\varepsilon\rangle}(y^{*}o F+z^{*}o H+\delta_{U\cap C}){(\bar{x})}\ne\emptyset.
	\end{eqnarray}
	\end{thm}

\begin{rem}
 To the best of our knowledge, there is no result on the existence of necessary optimality conditions of problem $(P)$ under $K$-convexlike  assumption. Therefore, Theorems \ref{t3} and \ref{t4} are new in the literature.
\end{rem}

\section{Acknowledgment}
%\begin{acknowledgment}
 F. Bozorgnia was  supported by the Portuguese National Science Foundation through FCT fellowship SFRH/BPD/33962/2009.

  The authors are very grateful to   anonymous referees for
carefully reading their manuscript and for several comments and
suggestions which helped them to improve the paper.

%\appendix  %This command ends the counting of sections.
%\section*{Appendix:  Instructions for Appendices}

%References
% BibTeX users  please use  \bibliographystyle{spmpsci_unsrt}. The option spmpsci_unsrt prints the references in JOTA format  in the order they are cited.
%Otherwise, please use the following:

\end{document}